\documentclass[12pt]{article}
\usepackage{amsthm}

\newtheorem{lemma}{Lemma}

\newtheorem{corollary}{Corollary}
\newtheorem*{obs*}{Observation}
\newtheorem{prop}{Proposition}

\newtheorem{definition}{Definition}
\newtheorem*{def*}{Definition}
\newtheorem*{lemma*}{Lemma}
\usepackage{amssymb}
\usepackage{mathtools}
\usepackage{graphicx}
\usepackage{caption}
\usepackage{enumitem}

\def \nmr {\begin{enumerate}}
	\def \enmr {\end{enumerate}}

\def \tmz {\begin{itemize}}
	\def \etmz {\end{itemize}}

\newcommand{\cG}{{\cal G}}
\newcommand{\cF}{{\cal F}}
\newcommand{\cB}{{\cal B}}
\newcommand{\cH}{{\cal H}}
\usepackage{microtype}
\usepackage{amsmath}
\newtheorem{thm}{Theorem}[section]
\newtheorem*{obs*[thm]}{Observation}
\newtheorem*{thm*}{Theorem}
\usepackage{amssymb}
\usepackage{mathtools}
\usepackage{graphicx}
\usepackage{caption}
\usepackage{enumitem}

\newenvironment{unnumbered}[1]{\trivlist \item [\hskip \labelsep {\bf
		#1}]\ignorespaces\it}{\endtrivlist}

\def \nmr {\begin{enumerate}}
	\def \enmr {\end{enumerate}}

\def \tmz {\begin{itemize}}
	\def \etmz {\end{itemize}}

 \title{Bipartite graphs with close domination and $ k $-domination numbers}

  \author{G\" ulnaz Boruzanl{\i} Ekinci$\/^{a}$ \quad and \quad Csilla Bujt\'{a}s$\/^{b}$ \\\\
  	$^{a}$ \small Department of Mathematics, Ege University,
  	 Izmir, Turkey\\
  	 \small{\tt  gulnaz.boruzanli@ege.edu.tr}, \\
  	$^{b}$ \small Faculty of Mathematics and Physics, University of Ljubljana, Slovenia \\
  	  	  	 \small {\tt csilla.bujtas@fmf.uni-lj.si}
  }
\date{\empty}

\begin{document}
	
	\maketitle
  
   \begin{abstract}
  Let $ k $ be a positive integer and let $ G $ be a graph with vertex set $ V(G) $. A subset $ D \subseteq V(G)$ is a $ k $-dominating set if every vertex outside $ D $ is adjacent to at least $ k $ vertices in $ D $. The $ k $-domination number  $ \gamma_k(G) $ is the minimum cardinality of a $ k $-dominating set in $ G $. For any graph $ G $, we know that $ \gamma_k(G) \geq \gamma(G)+k-2$ where $ \Delta(G)\geq k\geq 2 $ and this bound is sharp for every $ k\geq 2 $. In this paper, we characterize bipartite graphs satisfying the equality for $ k\geq 3 $ and present a necessary and sufficient condition for a bipartite graph to satisfy the equality hereditarily when $ k=3 $. We also prove that the problem of deciding whether a graph satisfies the given equality is NP-hard in general.
  \end{abstract}
  \medskip\noindent
  \textbf{Keywords:} Domination number, $ k $-domination number, Hereditary property, Vertex-edge cover, TC-number, Computational complexity
 \medskip
 
 \noindent
  \textbf{AMS Math.\ Subj.\ Class.\ (2010)}: 05C69, 05C75, 68Q25

\maketitle

\section{Introduction}

Let $ G $ be an undirected simple graph $ G $, where $ V(G) $ and $ E(G) $ denote the set of vertices and the set of edges of $ G $, respectively. For two vertices $ u,v \in V(G) $, $ u $ and $ v $ are \textit{neighbors} if they are adjacent, that is, if there is an edge $ e=uv \in E(G) $. The \textit{open neighborhood} of a vertex $ v $ is the set $ N_G(v) = \{u \in V(G): uv \in E(G)\} $ and the \textit{degree} of $ v $ is given by the cardinality of $ N_G(v) $. Two vertices $ u,v \in V(G) $ are \textit{false twins} if $ N_G(u) = N_G(v) $. Let $ \delta(G) $ and $ \Delta(G) $ denote the minimum and the maximum degree of $ G $, respectively. For a subset $ S\subseteq V(G) $, let $ G[S] $ denote the subgraph induced by $S$. An edge $ e=uv \in E(G) $ is \textit{subdivided} when it is deleted from $ G$ and a new vertex $ x $ is added with two new edges $ xu $ and $ xv $. Here, the vertex $ x $ is a \textit{subdivision vertex}. Let $ [r] $ denote  the set of integers $ \{1,\dots,r\} $ throughout the paper.

A subset $ D\subseteq V(G) $ is \textit{dominating} in $ G $ if every vertex of $ V(G) \setminus D $ has at least one neighbor in $ D $. Similarly, a subset $ D \subseteq V(G) $  is \textit{$ k $-dominating} in $ G $ if every vertex of $ V(G)\setminus D $ has at least $ k $ neighbors in $ D $. The \textit{domination number} $ \gamma(G) $ and the \textit{$ k $-domination number} $\gamma_k(G) $ of $ G $ are the minimum cardinalities of a dominating and a $ k $-dominating set of $ G$, respectively.

We say that a connected graph $G$ is a \emph{$(\gamma, \gamma_k)$-graph} if $\gamma_k(G)=\gamma(G) +k-2$ and $\Delta(G)\ge k$.
A connected graph $G$ is \emph{$(\gamma, \gamma_k)$-perfect} if  $\delta(G) \ge k$ and  every connected induced subgraph $H$ of $G$ with $\delta(H) \ge k$ satisfies the equality $\gamma_k(H)=\gamma(H) + k-2$.

Hypergraphs are set systems  that are conceived as a natural generalization of graphs.  A \emph{hypergraph} $H = (V,E)$ contains a finite set $V$ of vertices  together with a collection $E$ of nonempty subsets of $V$, called \emph{hyperedges} or simply \emph{edges}. The number of vertices, that is $|V|$, is called the \emph{order} of $H$. Throughout this paper, we suppose that  $|e| \ge 2 $ holds  for every $e\in E$. The \emph{degree} of a vertex $v$ in $H$, denoted by $d_H(v)$, is the number of edges containing the vertex $v$. The hypergraph $H$ is $k$-\emph{uniform} if every edge contains exactly $k$ vertices. Thus, every (simple) graph is a $2$-uniform hypergraph. A hypergraph $H'=(V', E')$ is an \emph{induced subhypergraph} of $H=(V,E)$ if $V' \subseteq V$ and $E'$ contains all edges $e \in E$ satisfying $e \subseteq V'$. We also use the notation $H[V']$ for the subhypergraph induced by $V'$. Given a collection $\cF$ of $k$-uniform hypergraphs, we say that a $k$-uniform hypergraph $H$ is \emph{$\cF$-free} if no induced subhypergraph of $H$ is isomorphic to any hypergraph contained in $\cF$. The complete $k$-uniform hypergraph of order $n$ is the hypergraph ${\cal K}_n^k=(V,E)$, where $|V|=n\ge k$ and $E$ contains all $k$-element subsets of $V$.

A set $T \subseteq V$  is a \emph{transversal}
(or \emph{vertex cover}) in the hypergraph $H=(V,E)$ if $|e\cap T|\ge 1$ holds for every edge $e\in E$.  The complement $V\setminus T$ of a transversal $T$ is called a \emph{(weakly) independent} vertex set. The minimum cardinality of a transversal and the maximum cardinality of a weakly independent vertex set in $H$ are denoted by $\tau(H)$ and $\alpha_w(H)$, respectively. The minimum number of edges that together cover every vertex of $H$ is the \emph{edge cover number} of $H$ and denoted by $\rho(H)$. For instance, in a complete $k$-uniform hypergraph ${\cal K}_n^k$, any $k-1$ vertices form a maximum weakly independent vertex set and any $n-k+1$ vertices form a minimum transversal. In particular, we have $\tau({\cal K}_n^k)=n-k+1$, $\alpha_w({\cal K}_n^k)=k-1$, and $\rho({\cal K}_n^k) = \lceil\frac{n}{k}\rceil$.

Two vertices $u$ and $v$ of $H$ are \emph{adjacent} if there is an edge $e$ of $H$ such that $\{x,y\}\subseteq e$. A hypergraph is \emph{connected} if for every two vertices $x$ and $y$ there exists a sequence $x=x_0,x_1, \dots, x_j=y$ such that $x_{i-1}$ and $x_i$ are adjacent vertices for every $i \in [j]$. The \emph{$2$-section graph} of $H=(V,E)$ is the graph $G$ on the same vertex set $V$ such that two vertices  form an edge in $G$ if and only if they are adjacent in $H$. The \emph{incidence graph} of the hypergraph $H=(V,E)$ is the bipartite graph $G'=(A \cup B, E')$ where the vertices in $A$ and $B$ represent the vertices and edges of $H$, respectively. Moreover, two vertices, $a\in A$ and $b\in B$ are adjacent in $G'$ if the vertex of $H$ that is represented by $a$ is contained in the hyperedge represented by $b$. Note that $H$ is a connected hypergraph, if and only if, its $2$-section graph is connected and, if and only if, its incidence graph is connected. 

\vspace{0.3cm}\noindent\textbf{Structure of the paper.} In Section~\ref{sec:2}, we first cite some previous results and prove some general lemmas. Then, for each $k \ge 2$, three graph classes  $\cB_k^*$, $\cB_k$ and  $\cG_k$ are defined such that every connected $(\gamma, \gamma_k)$-graph belongs to $\cG_k$. Moreover, in Section~\ref{sec:3}, we concentrate on the case of $k \ge 3$ and show that every connected bipartite  $(\gamma, \gamma_k)$-graph is contained in $\cB_k$. We also prove that a connected bipartite graph $ G \in \cB_k $ is a $(\gamma, \gamma_k)$-graph if and only if its $ \gamma_k $-simplified graph $ G^* \in \cB_k^* $ is a $(\gamma, \gamma_k)$-graph. Then, we present a characterization of  bipartite $(\gamma, \gamma_k)$-graphs in terms of the properties of the underlying hypergraph. In Section~\ref{sec:4}, we work on the hereditary version of the problem and give a characterization for all bipartite $(\gamma, \gamma_3)$-perfect graphs. Later, in Section~\ref{sec:5}, we prove that it can be decided in polynomial time whether a given bipartite graph is a $(\gamma, \gamma_k)$-graph, while the corresponding decision problem is NP-hard on $ \cG_k $.

\section{Preliminary Results on $ (\gamma, \gamma_k)$-graphs}
\label{sec:2}

Fink and Jacobson \cite{Fink85, Fink85-2} introduced  $ k $-domination in graphs as a generalization of the concept of domination. Motivated by this definition, related problems have been studied extensively by many researchers (see for example \cite{Bujtas2017, Caro1990-2, Caro1990, Favaron1988, Favaron2008, Hansberg2015, Hansberg2013, Shaheen2009}). For more details, we refer the reader to the books on domination by Haynes, Hedetniemi and Slater \cite{ DominationBook2, DominationBook1} and to the survey on $ k $-domination and $ k $-independence by Chellali \textit{et al.} \cite{ChellaliSurvey}. 

Fink and Jacobson \cite{Fink85} proved the following result on the relation between the domination number and the $ k $-domination number of $ G $. 

\begin{thm} \cite{Fink85}
	\label{thm:Fink}
	For any graph $ G $ with $ \Delta(G)\geq k\geq 2 $, $ \gamma_k(G) \geq \gamma(G)+k-2$.
\end{thm}

Although it is proved that the above inequality  is sharp for every $k\ge 2$, the characterization of graphs attaining the equality is still open, even for the small values of $ k $.  The corresponding characterization problem was studied in~\cite{Hansberg2015, Hansberg2016, Hansberg2008}. Recently, we considered a large class of graphs and gave a characterization for the members satisfying the equality $ \gamma_2(G) = \gamma(G)$. We also proved that it is NP-hard to decide whether this equality holds for a graph. Moreover, we gave a necessary and sufficient condition for a graph to satisfy $ \gamma_2(G) = \gamma(G)$ hereditarily \cite{manu}. Some
similar problems involving different domination-type graph and hypergraph invariants were considered for example in~\cite{Arumugam2013, Blidia2006, Dettlaff2016, Hartnell1995, Brause2016, Randerath1998}.


\begin{lemma}
	\label{lem:1}
	Let $D$ be a minimum $ k $-dominating set of a graph $G$. If  $ \gamma_k(G)=\gamma(G) +k-2$, then $\gamma(G[D])\geq \gamma_k(G)-(k-2)$. 
\end{lemma}
\begin{proof}
	Suppose, to the contrary, that $\gamma(G[D])\leq \gamma_k(G)-k+1$. Let $ S $ be a dominating set of cardinality $|D|-k+1$ in $ G[D] $. We claim that $ S $ is also a dominating set in $ G $. Indeed, as we removed $k-1$ vertices from the $k$-dominating set $D$, every vertex in $ V(G)\setminus D $ is still dominated by at least one vertex of $ S $. Note that $S$ dominates all the vertices in $ D $ by the choice of $S$. Since $ \gamma(G) \leq |S| = \gamma_k(G) - k +1 $, we get a contradiction. Thus, $\gamma(G[D])\geq \gamma_k(G)-k+2$. 
\end{proof}
Since $\gamma(G) \leq |V(G)|-\Delta(G)$ holds for every graph $G$, the previous lemma directly implies the following result obtained by Hansberg (see Theorem 5 in \cite{Hansberg2015}).
\begin{lemma} \cite{Hansberg2015}
	\label{lem:AH}
	Let $D$ be a minimum $ k $-dominating set of a graph $G$. If  $ \gamma_k(G)=\gamma(G) +k-2$, then $\Delta(G[D])\leq k-2 $.
\end{lemma}
\medskip

For each $k\ge 2$, we define three graph classes which will have crucial role in our study.
\begin{definition}
\label{Def:B_k*}
	Given an arbitrary connected $ k $-uniform hypergraph $ F $ with vertex set $ V(F) = D = \{v_1,\dots, v_s\} $ and edge set $E(F)=\{e_1, \dots ,e_p\}$, we define the following graph classes.
	\tmz
	\item[$ (i) $] Define a pair of vertices $ X_{i} = \{x_i^1, x_i^2\} $ for every hyperedge $ e_i \in E(F)$. Let $X=\bigcup_{i \in [p]}X_{i}$ and  let
	$$V(G_1)=X \cup V(F),  \quad E(G_1)=\{x_i^jv_\ell: \enskip v_\ell \in e_i, \enskip i \in [p], \enskip j \in [2], \enskip \ell \in [s]\}.$$
	The graph class $\mathcal{B}^*(F)$ contains only this graph $G_1$ which is also called the double incidence graph of $F$.
	
	\item[$(ii)$] The class $\cB(F)$ contains a graph $G_2$ if it can be obtained from the double incidence graph $G_1$  in the following way. We keep all vertices and edges of $G_1$. For each edge $e_i \in E(F)$, we create some (maybe zero) false twins of the vertex $x_i^1$. Further, if $S \subseteq V(F)$ induces a complete subhypergraph in $F$, then we may supplement $G_1$ with some (maybe zero) new vertices which are adjacent to all vertices in $S$. We denote by $Y$ the set of the new vertices that is $Y=V(G_2)\setminus V(G_1)$. Putting it the other way around, we say that $G_1$ is the \emph{$\gamma_k$-simplified graph} of $G_2$.
	
	\item[$(iii)$] The class $\cG(F)$ contains $G_3$ if it can be obtained from a graph $G_2\in \cB(F)$ by supplementing it with some (maybe zero) new edges inside $ D $ and $X\cup Y$. These edges can be chosen arbitrarily, but each $X_{i}$ must remain independent. 
	\etmz
	For every $ k\geq 2 $, the graph classes $\cB^*_k$, $\cB_k$, and  $\cG_k$ contains those graphs $G$ for which there exists a $ k $-uniform hypergraph $F$ such that $G$ belongs to $\cB^*_k(F)$, $\cB_k(F)$, and  $\cG_k(F)$, respectively. We say that $ F $  is the underlying hypergraph of $ G $ if $ G \in \cG(F) $.
\end{definition}
 It is clear that each member of $\cB_k$ is bipartite and that $\cB_k^* \subseteq \cB_k \subseteq \cG_k$ holds for every $k \ge 2$. In particular, for each $k$-uniform hypergraph $F$, we have exactly one graph in $\cB_k^*$  having $F$ as its underlying hypergraph, but there are infinitely many such graphs in $\cB_k$.   Note that in \cite{manu}, where $ (\gamma,\gamma_2) $-graphs were studied, a graph class $ \cG $ was introduced that exactly corresponds to the class $ \cG_2 $ defined here. Moreover, by the results in \cite{Arumugam2013, Hartnell1995, Lingas2018}, the class $\cB_2$ is the collection of those connected graphs which satisfy $\tau(G)=\gamma(G)$ and $\delta(G) \ge 2$.

Hansberg proved the following lemma which directly implies that the class $\cG_k$ contains all $(\gamma, \gamma_k)$-graphs.  
\begin{lemma} \cite{Hansberg2015}
	\label{lem:AH2}
	Let $ G $ be a $(\gamma, \gamma_k)$-graph for an integer $k \ge 2$ and  suppose that $ D $ is a minimum $ k $-dominating set of $ G $. Then, for every $ k $ vertices $ v_1,v_2, \dots, v_k$ from $D$, if $ \bigcap_{i=1}^k N(v_i) \neq \emptyset $, then there is a nonadjacent pair $ x,y \in V\setminus D $ such that $ N_G(x) \cap D =  N_G(y)\cap D = \{v_1,v_2, \dots, v_k\}$.
\end{lemma}
\begin{corollary} \label{lem:2} If $ G $ is a connected $ (\gamma,\gamma_k) $-graph with $\Delta(G) \ge k \ge 2$ and $D$ is a minimum $k$-dominating set of $G$, then there exists a $k$-uniform underlying hypergraph on the vertex set $D$ such that $ G\in \cG_k(F)$.
	\end{corollary}



\section{Bipartite  ($ \gamma,\gamma_k $)-graphs} \label{sec:3}

In this section we characterize connected bipartite graphs satisfying $\Delta(G) \ge k$ and $\gamma_k(G)=\gamma(G)+k-2$, for every $k \ge 3$. Observe that  $\gamma_k (G) \ge \gamma(G)$ holds  for every graph without imposing conditions on the vertex degrees. Moreover, assuming $k \ge 3$, the equality $\gamma_k(G)=\gamma(G)$ is satisfied if and only if $G$ consists of isolated vertices. Consequently, if $\Delta(G) \ge k \ge 3$ and if $G$ is disconnected, then  $\gamma_k(G)=\gamma(G)+k-2$ holds if and only if $G$ contains exactly one component which is a $(\gamma, \gamma_k)$-graph and all the further components are isolated vertices. Therefore, it is enough to concentrate on the connected $(\gamma, \gamma_k)$-graphs.

In the first subsection, we prove that all connected bipartite $(\gamma, \gamma_k)$-graphs belong to $\cB_k$ and it is enough to consider the graphs from the subclass $\cB_k^*$ where the underlying hypergraph uniquely represents the graph $G$. In the second subsection, we introduce a hypergraph invariant called `vertex-edge cover number'\ or shortly `TC-number'\ and derive the characterization of connected bipartite $(\gamma, \gamma_k)$-graphs via the properties of the underlying hypergraph.
\medskip

\subsection{Reducing the problem to $\gamma_k$-simplified graphs}
\begin{thm}
	\label{thmPartite}
	Let $k \ge 3$ and let $G$ be a connected bipartite graph that satisfies $\gamma_k(G)=\gamma(G)+k-2$ and $\Delta(G) \ge k$. Then, $G \in \cB_k$ and every minimum $ k $-dominating set corresponds to a partite class of $G$. 
\end{thm}
\begin{proof}
	Let $D$ be a minimum $ k $-dominating set of $G$. The degree condition implies that $V(G)\setminus D$ is not empty and, therefore, the underlying hypergraph $F$ contains at least one edge.  Let $e_i=\{v_1, \dots, v_k\}$ be an edge of $F$. Since all the vertices in $ e_i $ have common neighbors in $X_i$, they belong to the same partite class of $G$. This implies that the vertices of a connected component of $F$ are in the same partite class. If $F$ is connected, the same is true for the entire $D$. Moreover, since every vertex $x\in V(G)\setminus D$ has some neighbors in $D$, they all must be contained in the other partite class. It finishes the proof for the case when $F$ is connected.
	
	If $F$ is not connected, consider a nontrivial component $F_1$ and the remaining part $F_2=F-F_1$ of the underlying hypergraph. By Corollary~\ref{lem:2}, no vertex of $V(G)\setminus D$ can be adjacent to a vertex from $V(F_1)$ and to a vertex from $V(F_2)$ simultaneously. By our assumption, $G$ is connected. It is enough to consider the following two cases.
		
	\vspace{0.3cm}\noindent\textbf{Case 1.} There exists an edge $uv$ in $G[D]$ such that $u\in V(F_1)$ and $v\in V(F_2)$. 
	
	\noindent Since $F_1$ is connected and nontrivial, there is an edge $e_i \in E(F_1)$ containing $ u $. In this case, the $ k $ vertices from $ e_i \setminus \{u\} \cup \{v\} $ cannot have a common neighbor in $V(G)\setminus D$ and consequently, \[D'=D\setminus (e_i \setminus \{u\} \cup \{v\}) \cup \{x_i^1\},\]
	
	\noindent where $ x_i^1$  is from the pair $ X_i $ associated with $ e_i $. Observe that $ D' $ is a dominating set of $G$. As $|D'| <|D|-k+2$, this is a contradiction.

	\vspace{0.3cm}\noindent\textbf{Case 2.} There exists an edge $xy$ in $G[V(G)\setminus D]$ such that $x$ has neighbors from $V(F_1)$ and $y$ has neighbors from $V(F_2)$. 
	
	\noindent Consider $ k $ vertices, namely $v_1,\dots, v_k$ from $N(x)\cap V(F_1)$ and $ k $ vertices, namely $u_1,\dots,u_k$ from $N(y)\cap V(F_2)$ and then define $D'= D\setminus \{v_1,\dots, v_{k-1},u_1,\dots, \allowbreak u_{k-1}\}\cup \{x,y\}$. Observe that $D'$ is a dominating set of $G$, since $\{v_1,\dots, v_{k-1},u_1,\allowbreak \dots, u_{k-1}\}\cup \{x,y\}$ does not contain any edges of $F$. By our condition $ k\geq 3 $, it follows that  $|D'|= |D| - (2k-2) + 2 <|D|-k+2$. We have a contradiction again which completes the proof of the theorem.
	
\end{proof}

Note that Theorem \ref{thmPartite} does not hold for $ k=2 $. As an example, consider the graph $ H $ constructed by two vertex-disjoint copies of $ K_{2,3} $ by adding exactly one edge between them such that the maximum degree remains three. Observe that this graph is bipartite and satisfies the equality $ \gamma_2 (H) = \gamma(H) = 4$, but neither of the partite classes corresponds to a minimum dominating set or minimum $ 2 $-dominating set.

\begin{thm}
	\label{thm:simpl}
	Let $ G $ be a connected bipartite graph from $\cB_k $ and let $G^*$ be its $\gamma_k$-simplified graph from $\cB_k^*$, where $ k\geq 3 $. Then $\gamma_k(G)=\gamma(G)+k-2$ if and only if  $\gamma_k(G^*)=\gamma(G^*)+k-2$.
\end{thm}
\begin{proof}
	First, we prove that $\gamma_k(G)=\gamma_k(G^*)$. Since $D$ is supposed to be a minimum  $ k $-dominating set in $G$ and we did not delete any vertex from it when $G^*$ was obtained, $D$ is a $ k $-dominating set in $G^*$. Hence, $\gamma_k(G^*) \le |D|$. On the other hand, each vertex of $D$ can be $ k $-dominated in $ G^* $ by itself or by at least $ k $ neighbors from $X$. Since every vertex from $X$ is of degree $ k $, the $ k $-domination of all vertices in $D$ needs at least $|D|$ vertices in $G^*$. So, $\gamma_k(G^*)\ge |D|$ and $D$ is a minimum  $ k $-dominating set in $G^*$ as well. 
	
	Now, we prove that $\gamma(G)=\gamma(G^*)$. Let $Q'$ be a minimum dominating set in $G^*$ such that $|Q'\cap D|$ is maximum. Then, $Q'$ contains at most one vertex from each pair $\{x_i^1,x_i^2\}$, since otherwise $x_i^2$ can be replaced by a vertex of $e_i$. Since $x_i^1$ and $x_i^2$ are false twins, we may suppose that the vertex $x_i^2$ does not belong to $Q'$ for any $ i $. Note that $Q'\cap D$ is a transversal in the underlying hypergraph $F$. We claim that $Q'$ is a dominating set in $G$ as well. Indeed, all vertices in $D\cup X$ are dominated. Further, for any vertex $y\in Y$ we may consider $ k $ arbitrarily chosen neighbors from $D$, they must form a hyperedge $e_j$ in $F$. The vertex $x_j^2$ is not in $Q'$ but it is dominated by a vertex from $Q'\cap e_j$. This vertex dominates $y$ as well. This proves $\gamma(G) \leq |Q'| = \gamma(G^*)$.
	
	On the other hand, consider a minimum dominating set $Q$ in $G$ such that $|Q\cap D|$ is maximum. 
	Since every $X_i$ contains false twin vertices, we have $|Q\cap X_i|\leq 1$ for every $i$. If  $Q\cap X_i$  contains a vertex different from $x_i^1$, we may replace it with $x_i^1$ in the dominating set. Again, $Q\cap D$ must be a transversal in $F$. Now assume that $Q$ contains a vertex $y$ from $Y$. Let $z_1, \dots, z_\ell$ be those vertices from $D$ which are privately dominated by $y$ (they have no further neighbors in the dominating set). Note that all of $z_1, \dots, z_\ell$ are outside of $Q$. By Corollary~\ref{lem:2} and the definition of $F$, if $\ell \ge k$, then the set $\{z_1, \dots, z_\ell\}$ must contain at least one edge from $F$. This contradicts the fact that $Q$ is a transversal in $F$.
	If $\ell \le  k-1$, then there is an edge $e_i \in E(F)$ containing  $z_1, \dots, z_\ell$. As $x_i^1$ can dominate all these vertices  $z_1, \dots, z_\ell$ and the vertex which dominates $x_i^2$ in $Q$ also dominates $y$, the set $Q\setminus \{y\}\cup \{x_i^1\}$ is a dominating set of $G$. 
	
	Perform these changes while there is a vertex  from $Y$  in the dominating set. At the end, we have a minimum dominating set $Q^*\subseteq (D\cup X)$ in $G$. Clearly, $Q^*$ is a dominating set of $G^*$, and we have $\gamma(G^*)\leq  |Q^*| =\gamma(G)$.  This implies $\gamma(G)=\gamma(G^*)$ which completes the proof.	
\end{proof}


\subsection{Characterization via Underlying Hypergraphs}

\begin{definition}
For a hypergraph $H$, a set $S \subseteq V(H) \cup E(H)$  is a \emph{vertex-edge cover} or shortly a \emph{TC-set} of $H$ if $S\cap V(H)$ is a transversal (vertex cover) in $H$ and the edges in $S\cap E(H)$ together cover all vertices outside $S\cap V(H)$. The smallest cardinality of a TC-set in $H$ is called the \emph{TC-number} of $H$ and denoted by $tc(H)$.
\end{definition}

\begin{prop} 
	\label{prop:tc}
	For every hypergraph $H$ of order $n$, the following statements hold and the upper bounds given in $(i)$ and $(ii)$ are sharp.
\begin{itemize}
	\item[$(i)$] If $r$ is the smallest size of an edge in $H$, then $tc(H) \le n-r+2$.
	\item[$(ii)$] If $H$ is $k$-uniform, then $tc(H) \le n-k+2$.
	\item[$(iii)$] If $H$ is $2$-uniform, that is a simple graph, then $tc(H)=n$.
\end{itemize}
\end{prop}
\begin{proof}
Let $H$ be a hypergraph and $e\in E(H)$ an edge of minimum cardinality $r$. If $v$ is an arbitrary vertex from $e$, the $r-1$ vertices in $e\setminus \{v\}$ does not contain any edges of $H$. Hence, $T=(V(H)\setminus e) \cup \{v\}$ is a transversal and $T \cup \{e\}$ is a TC-set in $H$. Then, we have $tc(H) \le |T|+1 = n-(r-1)+1=n-r+2$ that proves $(i)$. From $(i)$, the statement $(ii)$ can be obtained as a direct consequence. Observe further that every complete $k$-uniform hypergraph gives a sharp example as every transversal of ${\cal K}_n^k$ contains at least $n-k+1$ vertices and, if it is not the entire vertex set, we also have to put an edge into the TC-set. Thus, $tc({\cal K}_n^k)= n-k+2$ holds for every $n \ge k \ge 2$.

Now, let $H$ be an arbitrary graph and let $T$ be its transversal set. Since every edge of $H$ intersects $V(H)\setminus T$ in at most one vertex, we cannot cover $V(H)\setminus T$ with less than $|V(H)\setminus T|=n-|T|$ edges. This gives $tc(H)\ge n$ and it is clear, or concluded from part $(ii)$, that $tc(H) \le n$. This proves $(iii)$.
\end{proof}

Concerning the extremal cases in part $(ii)$ of Proposition~\ref{prop:tc}, we prove the following.

\begin{thm}
	\label{thm:TC-extr}
 A $k$-uniform hypergraph $H$ of order $n$ satisfies $tc(H)=n-k+2$, if and only if, for every $\ell \le k-1$ and for every $\ell$  edges $e_{i_1}, \dots, e_{i_\ell}$ of $H$, the union $L=\bigcup_{j=1}^\ell e_{i_j}$ contains at most $\ell+k-2$ weakly independent vertices.
\end{thm}
\begin{proof}
	We prove the equivalence for the  negations.
	First suppose that there exist $\ell$ edges, say $e_{1}, \dots, e_{\ell}$, in $H$ such that $L=\bigcup_{j=1}^\ell e_{j}$ contains $s=\ell+k-1$ (weakly) independent vertices. Let $X=\{v_{1}, \dots, v_{s}\}$ be such an independent subset of $L$. Then, the set $T=V(H) \setminus X$ is a transversal of cardinality $n-s$ and the remaining vertices in $X$ can be covered by using the $\ell$ edges $e_{1}, \dots, e_{\ell}$. This TC-set contains $n-s+\ell=n-(\ell+k-1)+\ell=n-k+1$ elements and, as follows, we have $tc(H) \le n-k+1 <n-k+2$. This proves the first direction of the equivalence.
	
	Now we assume that $tc(H) \le n-k+1$ and  show the existence of $\ell \le k-1$ edges such that the union $L$ contains more than $\ell+k -2$ independent vertices. Consider a TC-set $S$ with $|S|=n-k+1$ and, renaming the edges and vertices if necessary, let $S\cap E(H)=\{ e_{1}, \dots, e_{\ell}\}$ and $X=V(H) \setminus (S \cap V(H))=\{v_{1}, \dots, v_{s}\}$ where $s=n-|S\cap V(H)|=n-(n-k+1-\ell)=k+\ell-1$. By the definition of a TC-set, $ X$ is an independent set in $H$ and all vertices in $X$ can be covered by the $\ell$ edges $e_{1}, \dots, e_{\ell}$. Therefore, the union $L$ of these $\ell$  edges contains a set $X$ of $s> \ell+k-2$ independent vertices. 
	
	Now it suffices to prove that we can ensure $\ell \le k-1$. Suppose that $\ell \ge k$ and the union $L$ of $e_{1}, \dots, e_{\ell}$ contains a set $X$ of $s=k+\ell -1$ independent vertices. Then, $\ell = s-k+1 \ge s-\ell +1$ holds and we may derive that $2\ell -1 \ge s$. By the pigeonhole principle, there exists an edge $e_q$, where $q \in [\ell]$, which contains at most one `private'\ vertex from $X$ that is, $$|(e_q \setminus \bigcup_{j \in [\ell], \: j\neq q} e_j ) \cap X| \le 1.$$
	Removing this edge $e_q$ from the list, we have $\ell -1$ edges such that their union contains at least $(\ell-1)+k-1$ independent vertices. If $\ell-1$ is still greater than $k-1$, we may repeat the procedure. At the end, we have $\ell' \le k-1$ edges such that their union contains $k + \ell'-1$ independent vertices. Consequently, it is enough to check the property for at most $k-1$ edges. 
		\end{proof}
	
	Motivated by Theorem~\ref{thm:TC-extr}, we define the following classes of hypergraphs.
	\begin{definition}
	For every $k\ge 3$, a $k$-uniform hypergraph $H$ belongs to the class $\cH_k$, if and only if, $\rho(H) \le k-1$ and $\alpha_w(H) \ge \rho(H)+k-1$.
	\end{definition}
	The condition  $\rho(H) \le k-1$ implies that for each $k$ and every $H \in \cH_k$, the order of $H$ is at most $k(k-1)$ and, consequently,  $\cH_k$ is a finite set of hypergraphs. Having the definition of $\cH_k$ in hand, we can formulate the following corollary.
	
	\begin{corollary}
		\label{cor:forbid-F}
	For each $k \ge 3$,  a $k$-uniform hypergraph $F$ satisfies $tc(F)=|V(F)|-k+2$ if and only if $F$ is $\cH_k$-free.
	\end{corollary}
\medskip

    The next theorem and its consequences give a characterization for connected bipartite $(\gamma, \gamma_k)$-graphs. In fact, these characterizations refer to the properties of the underlying hypergraph but, since we know that every bipartite $(\gamma, \gamma_k)$-graph $G$ belongs to a class $\cB_k(F)$, these results also characterize the structure of $G$.
    
    \begin{thm}
    	\label{thm:char-gamma-k}
   For each $k \ge 3$, a connected bipartite graph $G$ satisfies $\gamma_k(G)=\gamma(G)+k-2$, if and only if, $G \in \cB_k$ and the underlying hypergraph $F$ of $G$ satisfies $tc(F)=|V(F)|-k+2$. 
    \end{thm} 
\begin{proof}
If $G$ is a connected bipartite $(\gamma,\gamma_k)$-graph then, by Theorem~\ref{thmPartite}, we have $G \in \cB_k$. Further, by Theorem~\ref{thm:simpl}, $G$ is a $(\gamma, \gamma_k)$-graph if and only if  its $\gamma_k$-simplified graph $G^*$ is a $(\gamma,\gamma_k)$-graph as well. Note that $G$ and  $G^*$ admit the same underlying hypergraph $F$ where $|V(F)|= \gamma_k(G)= \gamma_k(G^*)$. Hence, it suffices to prove that $\gamma(G^*)=tc(F)$.

Consider now a minimum dominating set $Q$ of $G^*$ such that $|Q\cap V(F)|$ is maximum under this condition. First, observe that $Q$ does not contain two false twins from $V(G)\setminus V(F)$. Indeed,  $x_i^1 \in Q$ dominates itself and all vertices from the associated edge $e_i$ of $F$ and then, $x_i^2$ can be replaced by any vertex of $e_i$ in the dominating set. We may suppose, without loss of generality, that $x_i^2 \notin Q$ holds for each $e_i \in E(F)$. Since $N_{G^*}(x_i^2)=e_i$ and $|N_{G^*}(x_i^2) \cap Q| \ge 1$, we infer that $Q$ contains at least one vertex from each $e_i \in E(F)$. Thus, $Q\cap V(F)$ is a transversal of $F$. Further, if a vertex $v_j \in V(F)$ does not belong to $Q$, it is dominated by a vertex $x_s^1\in Q$ where the edge $e_s$ of the underlying hypergraph is incident with $v_j$. It means that the edge set
$$R=\{e_s: e_s \in E(F)    \mbox{ and } x_s^1 \in Q\}$$
covers all vertices of $V(F)$ which are outside $Q$. Then, $(Q \cap V(F))\cup R$ is a TC-set in $F$ and since $|R|$ equals the number of vertices in $Q \setminus V(F)$, the cardinality of this TC-set is $|Q|= \gamma(G^*)$. Therefore, $tc(F) \le |Q| =\gamma(G^*)$.
\medskip

To prove the other direction, suppose that $S'= T'\cup R'$ is a minimum TC-set of $F$ where $T'= S' \cap V(F)$ and $R'=S' \cap E(F)$. Now, determine the set $Q^* \subseteq V(G^*)$ as $Q^*=T' \cup R^*$ where $R^*$ consists of those vertices $x_i^1$ which represent the edges $e_i \in  R'$ that is,
$$R^*=\{x_i^1: N_{G^*}(x_i^1) \in R'\}.$$

Hence, we have $|Q^*|= |S'|= tc(F)$. Observe that, by the definition of the underlying hypergraph, the transversal $T'$ of $F$ dominates all vertices in $V(G) \setminus V(F)$ in the graph $G^*$. Moreover, each vertex $u$ of $V(F)$ which is outside $T'$ is covered by an edge $e_i \in R'$ in the underlying hypergraph and, therefore, $u$ is dominated by the vertex $x_i^1 \in R^*$ in $G$. We conclude that $Q^*$ is a dominating set of $G^*$ and we have $\gamma(G^*) \le |Q^*| =tc(F)$. This finishes the proof of the theorem.
\end{proof}

Theorem~\ref{thm:char-gamma-k}, Theorem~\ref{thm:TC-extr}, and Corollary~\ref{cor:forbid-F} immediately imply  other formulations of the characterization. 

\begin{corollary}
	\label{cor:TC-extr}
For each $k \ge 3$, a connected bipartite graph $G$ is a $(\gamma, \gamma_k)$-graph if and only if $G \in \cB_k$ and the underlying hypergraph $F$ of $G$ satisfies the following property: for every $\ell \le k-1$ and for every $\ell$  edges $e_{i_1}, \dots, e_{i_\ell}$ of $F$, the union $L=\bigcup_{j=1}^\ell e_{i_j}$ contains at most $\ell+k-2$ weakly independent vertices.
\end{corollary}

\begin{corollary}
	\label{cor:forbid-gamma-k}
	For each $k \ge 3$, a connected bipartite graph $G$ is a $(\gamma, \gamma_k)$-graph if and only if $G \in \cB_k$ and the underlying hypergraph of $G$ is $\cH_k$-free.
\end{corollary}

For the case of $k=3$, Corollary~\ref{cor:TC-extr} can be reformulated as follows.
\begin{corollary}
	\label{cor:2-2Prop}
A connected bipartite graph $G$ is a $(\gamma, \gamma_3)$-graph if and only if $G \in \cB_3$ and the underlying hypergraph $F$ satisfies the following property:
\begin{itemize}
\item[$(\star)$] Every four different vertices that can be split into two pairs of adjacent vertices induce at least one hyperedge in $F$.
\end{itemize}
	\end{corollary}
 \begin{prop}
 If $G$ is a connected bipartite $(\gamma, \gamma_3)$-graph and $D$ is a minimum $3$-dominating set of $G$, then  $G^2[D]$ is a threshold graph. 
\end{prop}
 \begin{proof}
 Suppose that $G$ and $D$ satisfies the conditions of the proposition and consider the underlying hypergraph $F$ with $V(F)=D$. Since $G \in \cB_3$, any two vertices $u,v\in D$ having distance $2$ in $G$ belong to a common hyperedge of $F$. In other words, if $uv \in E(G^2[D])$, then $uv$ is an edge in the $2$-section graph $H$ of $F$. Similar argumentation shows that the other direction is valid, too. Thus, $H=G^2[D]$. 
By Corollary~\ref{cor:2-2Prop}, the underlying hypergraph $F$ satisfies $(\star)$. This implies that  for every two vertex disjoint edges $vv'$ and $uu'$ of $H$, the vertex set $\{v,v',u,u'\}$ contains at least one hyperedge of $F$. We infer that there exist three vertices in $\{v,v',u,u'\}$ which induces a triangle in $H$. Checking all possible extensions of a $2K_2$ on four vertices, we obtain that $2K_2$, $P_4$ and $C_4$ are the forbidden induced subgraphs. Since  $(2K_2, P_4, C_4)$-free graphs are exactly the threshold graphs, the statement follows.
 \end{proof}

\section{Characterization of Bipartite $( \gamma, \gamma_3 )$-perfect Graphs} \label{sec:4}

In this section, we prove a characterization for bipartite $( \gamma, \gamma_3 )$-perfect graphs. First, we define the graph classes $\mathcal{S}_k$ and prove that all members of $\mathcal{S}_k$ are $(\gamma, \gamma_k)$-perfect  for every $k\ge 2$. Then, the main theorem of this section states that $\mathcal{S}_3$ is the set of all bipartite $(\gamma, \gamma_3)$-perfect graphs.

\begin{definition}
	\label{def:Sk}
  Let $k, r, i_1, \dots, i_r$ be integers which satisfy  $k \ge 2$, $r \ge 1$, and $i_j \ge k$ for every $j \in [r]$. We define $ S_k(i_1, \dots, i_r) $ as the graph that is obtained from the star $K_{1,r}$ in the following way. First replace the edges $e_1, \dots, e_r$ of the star with $i_1, \dots, i_r$ parallel edges, respectively, and subdivide each edge exactly once. Finally, supplement the graph by $k-2$ new vertices which are false twins with the center. For a fixed integer $k \ge 2$, let $ \mathcal{S}_k $ be the graph class that contains  all  $S_k(i_1, \dots, i_r) $ with $r \ge 1$ and $i_j \ge k$ for every $j \in [r]$.
\end{definition}
With this notation, $S_2(\ell)$ corresponds to $K_{2, \ell}$ whenever $\ell \ge 2$. In general, $S_k (\ell)$ is just $K_{k, \ell}$ if $\ell \ge k \ge 2$.  For another example that is not a complete bipartite graph see Fig.~\ref{fig:S3}.

	\begin{figure}[t]
\centerline{\includegraphics[width=0.5\textwidth]{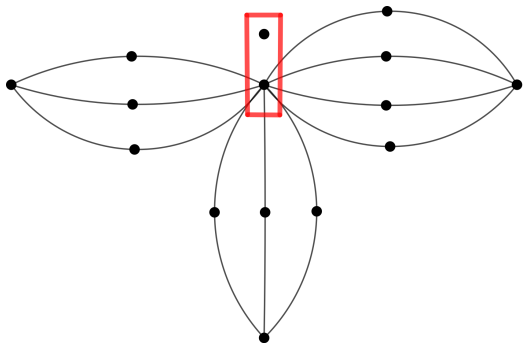}} 	
	\captionsetup{width=0.8\textwidth}
	\centering
	\caption{\protect An illustration for the graph $ S_3(3,3,4) $. The thick red box represents two false twin vertices.}
	\label{fig:S3}
\end{figure}

\begin{prop}
	\label{prop:Sk}
	If $ G \in \mathcal{S}_k $, then $ G $ is ($ \gamma,\gamma_k $)-perfect.
\end{prop}
\begin{proof}
	Consider a graph $G=S_k(i_1,\dots, i_r)$ from the graph class $\mathcal{S}_k$.  It is straightforward to check that the minimum vertex degree is $k$ and that $\gamma_k(G)= r+k-1$ and $\gamma(G)=r+1$ hold. Alternatively, the equality $\gamma_k(G)=\gamma(G)+k-2$ can be verified directly by Theorem~\ref{thm:char-gamma-k}. Thus, $G$ is a $(\gamma, \gamma_k)$-graph. Note that every subdivision vertex has degree $k$. Now, consider an induced subgraph $H$ of $G$ that satisfies  $\delta(H) \ge k$. Suppose that a vertex $u \in V(G)$ does not belong to $V(H)$. If $u$ is the center or its false twin  then, since $\delta(H) \ge k$,   no subdivision vertices can be present in $H$. This contradicts the degree condition.  If $u$ is a leaf in the original star then,  by the degree condition again, no neighboring subdivision vertices belong to $H$. If $u$ is a subdivision vertex, then consider the neighbor of $u$ which was a leaf in the star, say $u'$, and denote the degree $d_G(u')$ by $i_j$. When $u$ does not belong to $H$, we have two cases: either $u'$ and at least $k$ of its neighbors are still present in $H$; or  the entire $N_G[u']$ is omitted from $H$. We infer that every induced subgraph $H$ of $G$ with $\delta(H) \ge k$ is isomorphic to a graph $S_k(i_1',\dots , i_{r'}')$ where  $i_j' \ge k$ for all $j' \in [r']$. Therefore, $H$ is a $(\gamma, \gamma_k)$-graph again. This proves the $(\gamma, \gamma_k)$-perfectness of $G$.  
\end{proof}

As proved in \cite{manu}, ${\cal S}_2$ is the set of all  $(\gamma, \gamma_2)$-perfect graphs. Since each ${\cal S}_k$ contains only bipartite graphs, it is equivalent to saying that ${\cal S}_2$ is the set of all bipartite $(\gamma, \gamma_2)$-perfect graphs. Here we  prove an analogous statement for ${\cal S}_3$.

\begin{thm} \label{thm: perfect-3}
	$ G $ is a bipartite $( \gamma,\gamma_3 )$-perfect graph if and only if $ G \in \mathcal{S}_3 $.
\end{thm}
\begin{proof} 
	By Definition~\ref{def:Sk} and Proposition~\ref{prop:Sk}, every member of  $\mathcal{S}_3$ is  $(\gamma, \gamma_3)$-perfect and bipartite. To prove the other direction, suppose that $G$ is a bipartite $( \gamma,\gamma_3 )$-perfect graph, $D$ is a minimum $3$-dominating set and $F$ is the $3$-uniform underlying hypergraph of $G$ (with respect to $D$). Under these conditions, we first prove a couple of claims on the structure of $F$ and $G$.
	
 \begin{unnumbered}{Claim~A.}
 		Every edge of $F$ contains a vertex of degree $1$.
 \end{unnumbered}
\noindent
\textit{Proof of Claim A.} \enskip Suppose, to the contrary, that $e_1=\{u,u',u''\}$ is an edge in $F$ such that $d_F(u)\ge d_F(u')\ge d_F(u'')\ge 2$.	Let $Y_4$ be the set of vertices of degree at least $4$ in $V(G)\setminus D$ and let $Z= (N_G(u) \cap N_G(u') \cap N_G(u'')) \setminus \{x_1^1\}$. 

Consider the induced subgraph $H=G[V(G)\setminus (Y_4 \cup Z) ]$ of $G$. Since $G$ is bipartite and no vertex from $D$ was deleted, every vertex $y \in  V(H)\setminus D$ remains of degree $3$ in $H$. If $v$ is a vertex from $D$ such that  $d_F(v)=1$, then $v$ cannot be contained in a clique of order at least $4$ in $F$ and, by Corollary~\ref{lem:2}, $v$ cannot be adjacent to any vertex from $Y_4$.
By our assumption, $e_1$ does not contain the degree-1 vertex $v$ and therefore, $v$ is adjacent to none of the vertices in $Z$. Thus, $d_F(v)=1$ implies $d_{H}(v)=d_G(v) \ge 3$.  Now, consider a vertex $v \in D$ that satisfies $d_F(v) \ge 2$. This vertex $v$ is incident with at least two different edges in $F$, say $v$ is incident with $e_i$ and $e_j$. Observe that $d_G(v) \ge 4$ as  $v$ is adjacent to all vertices from $X_i\cup X_j$. Since $Y_4 \cup Z$ contains at most one vertex, namely $x_1^2$, from $X_i\cup X_j$, the degree of $v$ is still at least $3$ in $H$. We conclude $\delta(H)=3$. 

It is straightforward to show that, under our assumptions, $(D\cup \{x_1^1\}) \setminus \{u, u', u''\}  $ is a dominating set of $H$ and hence we have $\gamma(H) \le |D|-2$. We now prove that $\gamma_3(H)\ge |D|$ that  gives the desired contradiction. Consider an arbitrary $3$-dominating set $Q$ in $H$. Denote $|D\setminus Q|$ by $s$ and $|Q\cap (V(H) \setminus D)|$ by $\ell$. To $3$-dominate all vertices in $D\setminus Q$, we need at least $3s$ edges between the vertex sets $D\setminus Q$ and $Q\cap (V(H) \setminus D)$. On the other hand, we may have at most $3\ell$ edges between the two sets because every vertex in the second set is of degree $3$. Consequently, $\ell \ge s$ holds and this implies 
$$|Q|=|Q\cap D|+ |Q\cap (V(H) \setminus D)|= |D|-s+\ell \ge |D|$$
for every $3$-dominating set $Q$ of $H$. Therefore, we have $\gamma_3(H)\ge |D| >\gamma(H) +1$ for an induced subgraph of minimum degree $3$ that contradicts the $(\gamma, \gamma_3)$-perfectness of $G$. This contradiction finishes the proof of Claim A.

\medskip
\begin{unnumbered}{Claim~B.}
	Every vertex from $V(G) \setminus D$ is of degree $3$ in $G$.  
\end{unnumbered}
\noindent
\textit{Proof of Claim B.} \enskip By Claim A, there is no complete subhypergraph in $F$ on four (or more) vertices. Together with Corollary~\ref{lem:2} these imply that there exist no vertices of degree greater than $3$ in $V(G)\setminus D$. Since $\delta(G) \ge 3$, this results in $d_G(y) = 3$ for every $y \in V(G)\setminus D$ as stated.

\medskip
\begin{unnumbered}{Claim~C.}
 If three vertices form an edge in the underlying hypergraph $F$, then they have at least three common neighbors in $G$.
\end{unnumbered}
\noindent
\textit{Proof of Claim~C.} \enskip By Claim~A, every edge $e=\{u,u',u''\}$ of $F$ contains a vertex of degree 1. We may assume that $u$ is such a vertex in $e$. By the definition of the underlying hypergraph and by Claim B, all neighbors of the degree-1 vertex $u$ are associated with the edge $e$ that is, 
$$N_G(u)= \{y \in V(G)\setminus D: N_G(y)=\{u,u',u''\} \}.$$
Since $|N_G(u)|=d_G(u) \ge 3$, there exist at least three common neighbors of $u$, $u'$, and $u''$.

\medskip
\begin{unnumbered}{Claim~D.}
	$F$ does not have two edges that share exactly one vertex.
\end{unnumbered}
\noindent
\textit{Proof of Claim~D.} \enskip
First suppose that two edges, namely $e_i$ and $e_j$, of $F$ share exactly one vertex $w$. Let  $e_i=\{u,u',w\}$ and $e_j = \{v,v', w\}$.  By Claim~C  $u$, $u'$ and $w$ have at least three common neighbors, denote them by $x_i^1$, $x_i^2$ and $x_i^3$. Similarly, let the common neighbors of $v$, $v'$ and $v''$ be $x_j^1$, $x_j^2$ and $x_j^3$. Consider the subgraph $H$ of $G$ induced by these $11$ vertices. Observe that $H$ has minimum degree $3$,  $\gamma_3({H})=5$ and, as $w$, $x_i^1$, $x_j^1$ form a dominating set,  $\gamma(H) \le 3$. Thus, we have $\gamma_3(H) > \gamma(H)+1$ for a connected induced subgraph of $G$ with minimum degree 3. This contradicts our condition on the $(\gamma, \gamma_3)$-perfectness of $G$. Consequently, $F$ cannot contain two edges intersecting in exactly one vertex.

\medskip

Concerning the statement and proof of Claim~D, remark that even if $e_i \cup e_j$ contains further edges in $F$, the considered subgraph $H$ is an induced subgraph of $G$. 

\medskip

To complete the proof of Theorem~\ref{thm: perfect-3} we make the following observations. If $F$ contains only one edge, then $G$ is a $K_{3,\ell}$ that is isomorphic to $S_3(\ell)$ for an integer $\ell \ge 3$. So, we may assume that $|D|\ge 4$ and $F$ contains at least two edges. Since $G$ is a connected bipartite graph and $D$ is one of the partite classes, the underlying hypergraph $F$ must be connected as well. Thus, $F$ has two intersecting edges $e_i$ and $e_j$. By Claim~D, $e_i$ and $e_j$ share two vertices, say $v_0^1$ and $v_0^2$.  Claim~A implies that the third vertex of $e_i$ is of degree $1$ and the same is true for $e_j$. The connectivity of $F$ then requires that each edge of $F$ is incident with  $v_0^1$ and $v_0^2$ and contains a further private vertex. Therefore, the underlying hypergraph $F$ of a   $(\gamma,\gamma_3)$-perfect $G$ can always be obtained (up to isomorphism) in the following way:
$$V(F)=\{v_0^1, v_0^2, v_1, \dots , v_r\};$$
$$E(F)=\{e_1, \dots , e_r\} \quad \mbox{where} \quad e_i=\{v_0^1, v_0^2, v_i\} \enskip \mbox{for every $i \in [r]$.}$$
Then, $G$ can be constructed by assigning at least three, say $i_j$, new vertices to each edge $e_j$ of $F$  and making adjacent the vertices in $e_j$ to the new associated vertices. This clearly results in the graph $S_3(i_1,\dots, i_r) $ with $i_j\ge 3$ for all $j \in [r] $. Note that the condition $\delta(G)\ge 3$ implies that $F$ is not edgeless and consequently, $r \ge 1$ must hold. We conclude that every $(\gamma, \gamma_3)$-perfect graph belongs to $\mathcal{S}_3$.
\end{proof}	
\medskip

We show that the statement analogous to Theorem~\ref{thm: perfect-3} holds neither for $k=4$ nor for any even integer $k>4$. First, consider the lexicographic product $C_6 \circ \overline{K}_{2}$ that is, every vertex of the cycle $C_6$ is replaced with $2$ independent vertices. Then, $\gamma_4(C_6 \circ \overline{K}_{2})=6=\gamma(C_6 \circ \overline{K}_{2})+2$ and, since the graph is $4$-regular, every proper subgraph has a vertex of degree smaller than $4$. Thus, $C_6 \circ \overline{K}_{2}$ is $(\gamma, \gamma_4)$-perfect and bipartite but does not belong to $\mathcal{S}_4$. As an infinite class of similar examples we propose the following construction.

\vspace{0.3cm}\noindent\textbf{Example.} For each even integer $k=2\ell \ge 4$ consider the bipartite graph $G_k \in G_k^*$ obtained as the double incidence graph of the following underlying hypergraph $F_k$. The vertex set of $F_k$ is $V=W \cup U$ where $W=\{w_1, \dots, w_\ell\}$ and $U=\{u_1, \dots, u_{\ell+1}\}$; the edge set is $E=\{e_1, \dots , e_{\ell+1}\}$ where $e_i = V \setminus \{u_i\}$ for all $i\in [\ell+1]$. One can check that $F_k$ satisfies the condition given in Corollary~\ref{cor:TC-extr} and, therefore, $G_k$ is a $(\gamma, \gamma_k)$-graph. On the other hand, one can check that every proper subgraph of $G_k$ has minimum degree less than $k$. We may conclude that $G_k$ is a bipartite $(\gamma, \gamma_k)$-perfect graph but does not belong to $\mathcal{S}_k$. 

\section{Complexity Results} \label{sec:5}

\begin{prop}
	Let $k$ be a fixed integer with $k \ge 3$ and let $G $ be a bipartite graph with $\Delta(G) \ge k$. It can be decided in polynomial time whether the graph $ G $ satisfies the equality $ \gamma_k(G) = \gamma(G) +k-2$.	
\end{prop}

\begin{proof} If $G$ is a connected bipartite $(\gamma, \gamma_k)$-graph with $\Delta(G) \ge k$, then, by Theorem~\ref{thmPartite}, every minimum $k$-dominating set is a partite class of $G$ and $G \in \cB_k$. This can be checked efficiently and, assuming that the conditions are satisfied, the $k$-uniform underlying hypergraph $F$ can be determined in polynomial time as well. Then, by Corollary~\ref{cor:forbid-gamma-k}, it is enough to decide whether $F$ is $\cH_k$-free. Recall that $\cH_k$ is a finite set of $k$-uniform hypergraphs each of which is of order at most $k(k-1)$. Thus, this step can also be performed in polynomial time. 

Finally, note that if $G$ is disconnected, it satisfies $ \gamma_k(G) = \gamma(G) +k-2$, if and only if, $G$ contains one component which is a $(\gamma,\gamma_k)$-graph and the further components are isolated vertices. Hence, the statement remains true over the class of disconnected bipartite graphs. 
\end{proof}

\begin{thm}
	For each $ k \geq 2 $, it is NP-hard to decide whether the equality $ \gamma_k(G) = \gamma(G)+k-2 $ holds for $G$ over the class  $ \cG_k$.
\end{thm}

\begin{proof}
In order to prove that the corresponding decision problem is NP-hard, we construct a polynomial time reduction from 3-SAT problem, a classical NP-complete problem \cite{Garey1979}. Note that we adopt the approach we used in \cite{manu} to prove NP-hardness. Since the construction we need here is more general, we only present the construction explicitly and give a sketch of the rest of the proof, for the brevity. 

	
%
%
\begin{figure}[t]
	\centerline{\includegraphics[width=1.2\textwidth]{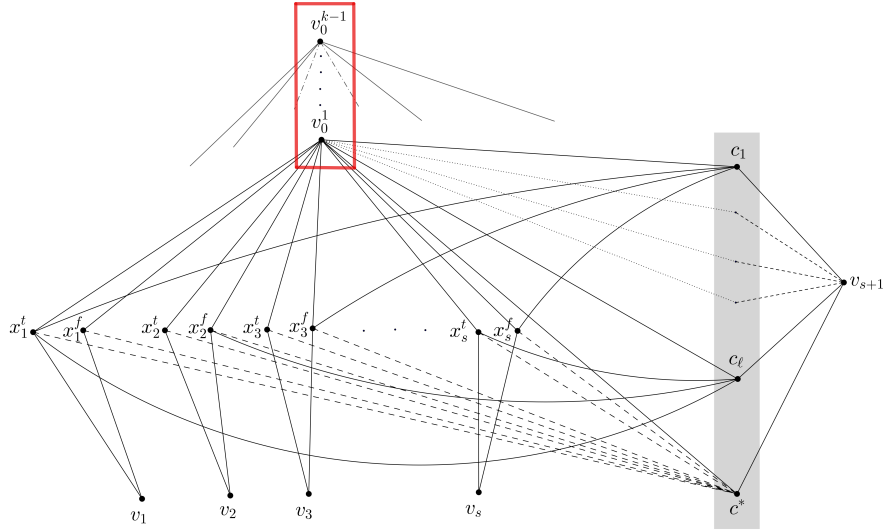}} 	
	\captionsetup{width=0.8\textwidth}
	\centering
	\caption{\protect An illustration of the construction for $ 3 $-SAT reduction: The clauses $ C_1 $ and $ C_\ell $ corresponding to the vertices $ c_1 $ and $ c_\ell $, resp., are $ C_1=(x_1 \vee \neg x_3 \vee \neg x_s) $ and  $ C_\ell=(x_1 \vee \neg x_2 \vee x_s) $. The thick red box represents $k-1$ independent vertices with the same open neighborhood.}
	\label{fig:construction}
\end{figure}
Let $ \mathcal{C} $ be a $ 3 $-SAT instance with clauses $ C_1,C_2,\dots, C_\ell $ over the Boolean variables $ X=\{x_1,x_2, \dots, x_s\} $. We may assume that for every three variables $ x_{i_1}, x_{i_2}, x_{i_3} $ there exists a clause $ C_j $, where $ j\in [\ell] $, such that $ C_j $ does not contain any of the variables $ x_{i_1}, x_{i_2}, x_{i_3} $ (neither in positive form, nor in negative form). Otherwise, the problem could be reduced to at most eight (separated) $ 2$-SAT problems, which are solvable in polynomial time.

We construct a graph $ G \in  \cG_k$, such that the given instance $ \mathcal{C} $ of $ 3 $-SAT problem is satisfiable if and only if $ \gamma(G) < \gamma_k(G)-k+2$. 

For every variable $ x_i $, we create three vertices $ \{x_i^t, x_i^f, v_i\} $ and then we add the edges $ x_i^tv_i $ and $ x_i^fv_i $. For every clause $ C_j \in \mathcal{C}$, we create a vertex $ c_j $, and if $ x_i $ is a literal in $ C_j $, then $ x_i^tc_j \in E(G) $; if $ \neg x_i $ is a literal in $ C_j $, then    $ x_i^fc_j \in E(G) $. Moreover, we add a vertex $ c^* $ and the edges $ c^*x_i^t  $ and $  c^*x_i^f $ for every $ i\in [s] $. We also add a vertex $ v_{s+1} $ and the edge set $  \{c_iv_{s+1}:1\le i\le\ell\} \cup \{c^*v_{s+1}\}$.  Finally, we add $ k-1 $ new vertices $ v^1_0, \dots, v^{k-1}_0$, such that each vertex $ v^r_0 $ is adjacent to every vertex in $ V(G)\setminus \{v_1,v_2\dots,v_{s+1}\} $ for $ r \in [k-1] $ (for an illustration of the construction, see Fig.~\ref{fig:construction}). The order of $ G $ is obviously $ 3s+\ell +k+1 $ and this construction can be done in polynomial time. Note that $ G \in \cG(F)$ and any two hyperedges in $ F $ share $ k-1 $ vertices, namely $ v_0^1, \dots, v_0^{k-1} $. It is straightforward to prove that $\{v_1, \dots, v_{s+1}, v_0^1, \dots, v_0^{k-1}\}$ is a minimum $k$-dominating set in $G$ and, therefore, $\gamma_k(G)=k+s$.

In order to finish the proof, it suffices to show that $ \mathcal{C}  $ is satisfiable if and only if $ \gamma(G) < \gamma_k(G)-k+2 $. First, assume that $ \mathcal{C} $ is satisfiable and let $ \varphi: X \rightarrow \{t,f\} $ be  the corresponding truth assignment. The set $ D' = D_1 \cup D_2 \cup \{c^*\} $ is a dominating set, where  $D_1=\bigcup_{i\in [s] }\{x_i^t :\varphi(x_i)=t\} $ and  let $D_2 = \bigcup_{i\in [s] }\{x_i^f : \varphi(x_i)=f\}$. Since $ |D'|=s+1 $, we have $ \gamma(G) < \gamma_k(G)-k+2=(k+s)-k+2 $ for $ k\geq 2 $.

For the other direction, assume that $ \gamma(G) < \gamma_k(G)-k+2 $ and consider a dominating set $D'$ such that $ |D'|\leq s+1 $. In order to dominate $ v_i $, the set $ D' $ contains at least one vertex from the set $ \{x_i^t, x_i^f, v_i\} $, for each $ i\in [s] $. Similarly, to dominate $ v_{s+1} $, the set $ D' $ contains at least one vertex from the set $ \{c_1,c_2,\dots,c_\ell,c^*,v_{s+1}\}$. Since $ |D'|\le s+1 $, we have $|D' \cap \{x_i^t, x_i^f, v_i\}| = 1 $ for every $ i\in [s] $. Moreover, $ |D'\cap \{c_1,c_2,\dots,c_\ell,c^*,v_{s+1}\}| = 1$ and $ \{v_0^1, \dots, v_0^{k-1}\} \cap D' = \emptyset $. There are three cases to consider and for a detailed discussion, we refer the reader to \cite{manu}.
\begin{itemize}
	\item[$ (i) $] If $ v_{s+1} \in D' $ then, to dominate $x_i^t$ and $x_i^f$ for each $i \in [s]$, every $v_i$ must be contained in $D'$. Hence,  the vertices $ \{v_0^1, \dots, v_0^{k-1}\} $ are not dominated by a vertex from $ D' $, a contradiction.
	\item[$ (ii) $] If $ c_j \in D' $ for some $ j \in [\ell] $, then for all but at most three $i \in [s]$ we have  $v_i \in D'$. Then, there is a vertex $ c_q $, such that $ q \in [\ell] $ and $ c_q $ is not dominated by a vertex from $ D' $, a contradiction. 
	\item[$(iii) $] If $ c^* \in D' $, then $c_j$ must be dominated by a vertex $x_i^t$ or $x_i^f$ for every $j \in [\ell]$. Now, consider the truth assignment $ \varphi: X \rightarrow \{t,f\} $  where $\varphi(x_i)=t$ if and only if $x_i^t \in D'$ and observe that $\varphi$ satisfies the 3-SAT instance $ \mathcal{C} $.
\end{itemize}
This finishes the proof. 
\end{proof}


\section*{Acknowledgements}
  This research was started when the second author visited the Ege University in Izmir, Turkey; the authors thank the financial support of T\" UBİTAK under the grant  BİDEB 2221 (1059B211800686). The second author also acknowledges the financial support from the Slovenian Research Agency under the project N1-0108.

\bibliographystyle{siamnodash}
\bibliography{references}

\begin{thebibliography}{10}

\bibitem{Arumugam2013}
S~Arumugam, Bibin~K Jose, Csilla Bujt{\'a}s, and Zsolt Tuza.
\newblock Equality of domination and transversal numbers in hypergraphs.
\newblock {\em Discrete Applied Mathematics}, 161(13):1859--1867, 2013.

\bibitem{Blidia2006}
Mostafa Blidia, Mustapha Chellali, and Teresa~W Haynes.
\newblock Characterizations of trees with equal paired and double domination
  numbers.
\newblock {\em Discrete Mathematics}, 306(16):1840--1845, 2006.

\bibitem{Bujtas2017}
Csilla Bujt{\'a}s and Szil{\'a}rd Jask{\'o}.
\newblock Bounds on the $ 2 $-domination number.
\newblock {\em Discrete Applied Mathematics}, 242:4--15, 2018.

\bibitem{Caro1990-2}
Yair Caro.
\newblock On the $ k $-domination and $ k $-transversal numbers of graphs and
  hypergraphs.
\newblock {\em Ars Combinatoria}, 29:49--55, 1990.

\bibitem{Caro1990}
Yair Caro and Yehuda Roditty.
\newblock A note on the $ k $-domination number of a graph.
\newblock {\em International Journal of Mathematics and Mathematical Sciences},
  13(1):205--206, 1990.

\bibitem{ChellaliSurvey}
Mustapha Chellali, Odile Favaron, Adriana Hansberg, and Lutz Volkmann.
\newblock $ k $-domination and $ k $-independence in graphs: A survey.
\newblock {\em Graphs and Combinatorics}, 28(1):1--55, 2012.

\bibitem{Dettlaff2016}
Magda Dettlaff, Magdalena Lema{\'n}ska, Gabriel Semani{\v{s}}in, and Rita
  Zuazua.
\newblock Some variations of perfect graphs.
\newblock {\em Discussiones Mathematicae Graph Theory}, 36(3):661--668, 2016.

\bibitem{manu}
G{\"u}lnaz~Boruzanl{\i} Ekinci and Csilla Bujt{\'a}s.
\newblock On the equality of domination number and $2$-domination number.
\newblock {\em arXiv preprint arXiv:1907.07866}, 2019.

\bibitem{Favaron1988}
Odile Favaron.
\newblock $ k $-domination and $ k $-independence in graphs.
\newblock {\em Ars Combinatoria}, 25C:159--167, 1988.

\bibitem{Favaron2008}
Odile Favaron, Adriana Hansberg, and Lutz Volkmann.
\newblock On $ k $-domination and minimum degree in graphs.
\newblock {\em Journal of Graph Theory}, 57(1):33--40, 2008.

\bibitem{Fink85}
J.~F. Fink and M.~S. Jacobson.
\newblock $ n $-domination in graphs.
\newblock {\em Graph Theory with Applications to Algorithms and Computer
  Science}, \:283--300, 1985.

\bibitem{Fink85-2}
J.~F. Fink and M.~S. Jacobson.
\newblock On $ n $-domination, $ n $-dependence and forbidden subgraphs.
\newblock {\em Graph Theory with Applications to Algorithms and Computer
  Science}, \:301--311, 1985.

\bibitem{Garey1979}
Michael~R Garey and David~S Johnson.
\newblock {\em Computers and Intractibility: A Guide to the Theory of
  NP-Completeness}.
\newblock WH Freeman and Co., New York, 1979.

\bibitem{Hansberg2015}
Adriana Hansberg.
\newblock On the $ k $-domination number, the domination number and the cycle
  of length four.
\newblock {\em Utilitas Mathematica}, 98:65--76, 2015.

\bibitem{Hansberg2013}
Adriana Hansberg and Ryan Pepper.
\newblock On $ k $-domination and $ j $-independence in graphs.
\newblock {\em Discrete Applied Mathematics}, 161(10):1472--1480, 2013.

\bibitem{Hansberg2016}
Adriana Hansberg, Bert Randerath, and Lutz Volkmann.
\newblock Claw-free graphs with equal $ 2 $-domination and domination numbers.
\newblock {\em Filomat}, 30(10):2795--2801, 2016.

\bibitem{Hansberg2008}
Adriana Hansberg and Lutz Volkmann.
\newblock On graphs with equal domination and $ 2 $-domination numbers.
\newblock {\em Discrete Mathematics}, 308(11):2277--2281, 2008.

\bibitem{Hartnell1995}
Bert Hartnell and Douglas~F Rall.
\newblock A characterization of graphs in which some minimum dominating set
  covers all the edges.
\newblock {\em Czechoslovak Mathematical Journal}, 45(2):221--230, 1995.

\bibitem{DominationBook2}
Teresa~W Haynes, Stephen Hedetniemi, and Peter Slater.
\newblock {\em Domination in graphs: Advanced Topics}.
\newblock Marcel Dekker, 1997.

\bibitem{DominationBook1}
Teresa~W Haynes, Stephen Hedetniemi, and Peter Slater.
\newblock {\em Fundamentals of domination in graphs}.
\newblock CRC Press, 1998.

\bibitem{Brause2016}
Marcin Krzywkowski, Michael~A Henning, and Christoph Brause.
\newblock A characterization of trees with equal $ 2 $-domination and $ 2
  $-independence numbers.
\newblock {\em Discrete Mathematics \& Theoretical Computer Science}, 19, 2017.

\bibitem{Lingas2018}
Andrzej Lingas, Mateusz Miotk, Jerzy Topp, and Pawe{\l} {\.Z}yli{\'n}ski.
\newblock Graphs with equal domination and covering numbers.
\newblock {\em arXiv preprint arXiv:1802.09051}, 2018.

\bibitem{Randerath1998}
Bert Randerath and Lutz Volkmann.
\newblock Characterization of graphs with equal domination and covering number.
\newblock {\em Discrete Mathematics}, 191(1-3):159--169, 1998.

\bibitem{Shaheen2009}
Ramy~S Shaheen.
\newblock Bounds for the 2-domination number of toroidal grid graphs.
\newblock {\em International Journal of Computer Mathematics}, 86(4):584--588,
  2009.

\end{thebibliography}

\end{document}